\documentclass[11pt]{amsart}

\usepackage{amscd}
\usepackage{xypic}  
\usepackage{amssymb}
\usepackage{amsthm}
\usepackage{epsfig}
\usepackage{paralist}
\usepackage{url}
\usepackage{caption}
\usepackage[hidelinks=true,hypertexnames=false]{hyperref} 

\usepackage{tikz}

\usepackage[all,cmtip]{xy}
\usepackage{amsmath}
\usepackage{color}

\newtheorem{thm}{Theorem}[section]  

\newtheorem{lemma}[thm]{Lemma}
\newtheorem{prop}[thm]{Proposition}

\newtheorem{claim}{Claim}[thm]

\theoremstyle{definition}

\newtheorem{remark}[thm]{Remark}

\newtheorem{question}[thm]{Question}

 \newtheorem{defn}[thm]{Definition}

\def\codim{\operatorname{codim}}

\def\c1{\operatorname{c_1}}
\def\c2{\operatorname{c_2}}

\def\Sym{\operatorname{Sym}}

\def\PP{{\mathbb P}}

\def\DD{{\mathbb D}}

\def\M{{\mathcal M}}
\def\N{{\mathcal N}}
\def\O{{\mathcal O}}
\def\I{{\mathcal J}}

\def\E{{\mathcal E}}
\def\T{{\mathcal T}}

\def\F{{\mathcal F}}

\def\C{{\mathcal C}}

\def\Y{{\mathcal Y}}
\def\X{{\mathcal X}}
\def\FF{{\mathbb F}}
\def\e{\mathfrak{e}}
\def\s{\mathfrak{s}}
\def\f{\mathfrak{f}}
\def\c{\mathfrak{c}}

\def\cong{\simeq}

\def\leq{\leqslant}
\def\geq{\geqslant}

\def\+{\oplus}               
\def\*{\otimes}                  

\def\Bl{\operatorname{Bl}}

\def\Sing{\operatorname{Sing}}

\def\Bl{\operatorname{Bl}}

\begin{document}

\title{Severi varieties on blow--ups of the symmetric square of an elliptic curve} 

\author[C.~Ciliberto]{Ciro Ciliberto}
\address{Ciro Ciliberto, Dipartimento di Matematica, Universit{\`a} di Roma Tor Vergata, Via \
della Ricerca Scientifica, 00173 Roma, Italy}
\email{cilibert@mat.uniroma2.it}

\author[T.~Dedieu]{Thomas Dedieu}
\address{Thomas Dedieu,                                                                     
Institut de Math{\'e}matiques de Toulouse--UMR5219,                                          
Universit{\'e} de \linebreak Toulouse--CNRS,                                                            
UPS IMT, F-31062 Toulouse Cedex 9, France}
\email{thomas.dedieu@math.univ-toulouse.fr}

\author[C.~Galati]{Concettina Galati}
\address{Concettina Galati, Dipartimento di Matematica e Informatica, Universit{\`a} della \linebreak Calabria, via P. Bucci, cubo 31B, 87036 Arcavacata di Rende (CS), Italy}
\email{galati@mat.unical.it}

\author[A.~L.~Knutsen]{Andreas Leopold Knutsen}
\address{Andreas Leopold Knutsen, Department of Mathematics, University of Bergen, Postboks 7800,
5020 Bergen, Norway}
\email{andreas.knutsen@math.uib.no}

\date{\today}

\begin{abstract}  We prove that certain Severi varieties of nodal curves of positive genus  on general blow--ups of the twofold symmetric product of a general elliptic curve are non--empty and smooth of the expected dimension. This result, besides its intrinsic value,  is an important  preliminary step for the proof of nonemptiness of Severi varieties on general Enriques surfaces in \cite{CDGK}.  

\end{abstract}

\keywords{Severi varieties, nodal curves, degenerations, elliptic ruled surfaces}
\subjclass[2010]{14D06,14H10, 14H20}

\maketitle

\section{Introduction}

Let $S$ be a smooth, projective complex surface and $\xi \in {\rm Num}(S)$. Let 
$$p_a(\xi)=\frac{1}{2}\xi \cdot (\xi+K_S)+1$$ 
be the \emph{arithmetic genus} of $\xi$.  If $L$ is a line bundle or divisor on $S$ with class $\xi$ in  ${\rm Num}(S)$ we set $p_a(L)=p_a(\xi)$. We  denote by $V^\xi(S)$ the  locus in the  Hilbert scheme   of $S$ parametrizing the  curves $C$ on $S$ such that the class of $\mathcal O_S(C)$ in ${\rm Num}(S)$ coincides with $\xi$. Assume that $L$ is a line bundle or divisor on $S$ with class $\xi$ in  ${\rm Num}(S)$ and that $p_a(L)\geq 0$. For any integer $\delta$ satisfying $0 \leq \delta \leq p_a(\xi)$ we denote by
$V^\xi_{\delta}(S)$ the {\it Severi variety} parametrizing irreducible $\delta$--nodal curves contained in $V^\xi(S)$.  This is a possibly empty locally closed set in $V^\xi(S)$. 

Let $V$ be an irreducible component of $V^\xi(S)$ and, for any $\delta$ such that $0 \leq \delta \leq p_a(\xi)$, let  $V_\delta$ be an irreducible component of $V\cap V^\xi_{\delta}(S)$. It is well known that 
\begin{equation}\label{eq:expdim1a}
\dim(V_\delta)\geq \dim (V)-\delta,
\end{equation}
where the right hand side is called the {\it expected dimension} of $V_\delta$. Moreover if the equality holds in \eqref {eq:expdim1a}, then,  for all $0\leq \delta' \leq \delta$,  the
closure of the intersection of $V^\xi_{\delta'}(S)$ with $V$ contains $V_\delta$,  and any of its components whose closure contains $V_\delta$ has the expected dimension $\dim(V)-\delta'$ (see \cite[Thm. 6.3]{ser2}).

Severi varieties were introduced by Severi in Anhang F of \cite{sev}, where he proved  that all Severi varieties
of irreducible $\delta$-nodal curves of degree $d$ in $\PP^2$ are nonempty and smooth of the expected dimension. Severi also claimed irreducibility of such varieties, but his proof contains a gap. The irreducibility was proved by Harris \cite{Ha} more than 60 years later.

Severi varieties on other surfaces have received much attention in recent years, especially in connection with enumerative formulas computing their degrees. 

Nonemptiness is known to hold for all Severi varieties associated to big and nef classes on Del Pezzo surfaces (as well as rational surfaces under certain assumptions) by \cite[Thms. 3-4]{GLS} and for Hirzebruch surfaces, a result implicitly contained in \cite[\S 3]{Ta}.
In both cases of Del Pezzo and Hirzebruch surfaces, all Severi
varieties are smooth of the expected dimensions, cf., e.g. \cite[Lemma
2.9]{Ta} or \cite[p.~45]{CS}. Moreover, all Severi varieties of
Hirzebruch surfaces are irreducible \cite{Ty},
and Severi varieties parametrizing rational curves on Del Pezzo
surfaces of degrees $\geq 2$ are irreducible as well \cite{Te}; the
same holds true for {\it general} Del Pezzo surfaces of degree one,
except for the Severi variety parametrizing rational curves in the  anticanonical class \cite[Cor.~6.4]{Te2}.

On a general primitively polarized $K3$ surface $(S,\xi)$, all Severi varieties
$V^{n\xi}_{\delta}(S)$, where $0 \leq \delta \leq p_a(n\xi)$, are
nonempty by a result of Mumford \cite{MM} if $n=1$ and of Chen
\cite{chen} for all $n$; moreover, all components are always smooth of
the expected dimension $p_a(n\xi)-\delta$
\cite{Ta2,CS}.
The irreduciblity question (for $\delta < p_a(n\xi)$) has been the
object of much attention,
see \cite{keilen,kemeny,CDirrid,ballico,Dballico},
and was recently solved in the
case $n=1$ for all $\delta \leq p_a(\xi)-4$
in the preprint \cite{BL-C}.

Similarly, on a general primitively polarized abelian surface $(S,\xi)$, all Severi varieties
$V^{n\xi}_{\delta}(S)$, where $0 \leq \delta \leq p_a(n\xi)-2$, are nonempty (by \cite{KLM} if $n=1$ and \cite{KL} in general) and smooth of the expected dimension $p_a(n\xi)-\delta$ \cite{LS}. Irreducibility does not hold: the various irreducible components in the case $n=1$ have  been described by Zahariuc \cite{Za}.

Very little is known on other surfaces, where problems such as nonemptiness, smoothness, dimension and irreducibility are regarded as very hard. In particular,
Severi varieties may have unexpected behaviour: examples are given in \cite{CC} of surfaces of general type with reducible Severi varieties, and also with components of  dimension different from the expected one.

In this paper we consider the case of blow--ups of a particular type of ruled surface over an elliptic curve.

Let $E$ be a general smooth irreducible projective curve of genus one and set $R:=\Sym^2(E)$. Let $\widetilde{R}$ be the blow--up of $R$ at any finite set of general points. Our main result in this paper shows that Severi varieties of a large class of line bundles on $\widetilde{R}$ are well-behaved:

\begin{thm} \label{thm:R11}  In the above setting, let $L$ be a line bundle on
  $\widetilde{R}$ verifying condition $(\star)$ (cf. Definition \ref{def:good}) and let $\xi$ be the class of $L$ in ${\rm Num}(\widetilde{R})$.
  Let $\delta$ be an integer satisfying $0 \leq \delta <p_a(L)$.
  Then $V_{\delta}^\xi(\widetilde{R})$ is nonempty and smooth with all components of the expected dimension $- L \cdot K_{\widetilde{R}}+p_a(L)-\delta-1$. 
\end{thm}

The statement about smoothness and dimension follows from standard arguments of deformation theory, once non--emptiness has been proved, cf. Proposition \ref{prop:sed} below. 
Moreover we remark that, by what we said above, it suffices to prove Theorem \ref {thm:R11} for the maximal number of nodes, i.e., $\delta= p_a(L)-1$. This will follow from Proposition \ref {prop:R11} below, which treats the special case in which the blown--up points are in a special position. In \cite{CDGK} we will make use of Proposition \ref {prop:R11} in order to prove nonemptiness of Severi varieties on Enriques surfaces. 
The question of smoothness and dimension of Severi varieties on Enriques surfaces has been treated in  \cite{indam}.

The irreducibility question for $V_{\delta}^\xi(\widetilde{R})$ is not treated in this paper; thus, we pose:
\begin{question}
  Are the varieties $V_{\delta}^\xi(\widetilde{R})$ from Theorem \ref{thm:R11} irreducible?
\end{question}

The paper is organised as follows. In \S \ref {sec:prel} we recall some preliminaries concerning  twofold symmetric products  of elliptic curves. Section \ref {sec:deg} is devoted to recalling 
a degeneration of the symmetric product of a general elliptic curve studied in \cite{ck}. In \S \ref {sec:lemmas} we construct certain families of curves on some blow--ups of the projective plane that turn out to be useful in the proof of Proposition \ref {prop:R11}, which is proved by degeneration in \S \ref {sec:proof}.

\vspace{0.3cm} {\it Acknowledgements.} The authors acknowledge funding
from MIUR Excellence Department Project CUP E83C180 00100006 (CC),
project FOSICAV within the  EU  Horizon
2020 research and innovation programme under the Marie
Sk{\l}odowska-Curie grant agreement n.  652782 (CC, ThD),
 GNSAGA of INDAM (CC, CG), the Trond Mohn 
Foundation Project ``Pure Mathematics in Norway'' (ALK, CG) and grant 261756 of the Research Council of Norway (ALK).

\section{The twofold symmetric product of an elliptic curve}\label{sec:prel}

Let $E$ be a smooth irreducible projective  elliptic curve. Denote by $\+$ (and $\ominus$) the group operation on $E$  and by $e_0$ the neutral element. Let $R:=\Sym^2(E)$ and
$\pi: R \to E$ be the (Albanese)  morphism sending $x+y$ to $x\+ y$.
We denote the fiber of $\pi$ over a point $e \in E$ by 
\[ \f_e:=\pi^{-1}(e)=\{ x+y \in \Sym^2(E) \; | \; x\+ y=e \; \mbox{(equivalently,} \; 
x + y \sim e+e_0)\},\]
(where $\sim$ denotes linear equivalence), which is the $\PP^1$ defined by the linear system $|e+e_0|$. We denote the algebraic equivalence class of the fibers by $\f$. 

For each $e \in E$ we define the curve $\s_e$ (called $D_e$ in \cite{CaCi})
as the image of the section $E \to R$ of the Albanese morphism mapping $x$ to $e+ (x \ominus e)$.
We let $\s$ denote the algebraic equivalence class of these sections, which are the ones with minimal self-intersection, namely $1$, cf. \cite{CaCi}.  One has
$$
K_R \sim -2\s_{e_0}+\f_{e_0}.
$$

Let $y_1,\ldots,y_n \in R$ be distinct points and let $\widetilde{R}:=\Bl_{y_1,\ldots,y_n}(R) \to R$ denote the blow--up of $R$ at $y_1,\ldots,y_n$, with exceptional divisors $\e_i$ over $y_i$.
We denote the strict transforms of $\s$ and $\f$ on $\widetilde{R}$ by the same symbols.

 \begin{defn} \label{def:good}
  A line bundle or Cartier divisor $L$ on $\widetilde{R}$ is said to verify condition $(\star)$ if it is of the form $L\equiv \alpha\s+\beta \f-\sum_{i=1}^n \gamma_i \e_{i}$ (where $\equiv$ denotes numerical or, equivalently, algebraic equivalence), with $\alpha$, $\beta$, $\gamma_1,\ldots, \gamma_n$ are integers such that:\\
  \begin{inparaenum}
\item [(i)] $\alpha\geq 1$, $\beta \geq 0$;\\
\item [(ii)] $\alpha\geq \gamma_i$ for $i=1,\ldots, n$;\\
\item [(iii)] $\alpha+\beta\geq \sum_{i=1}^n\gamma_i$;\\ 
\item [(iv)] $\alpha+2\beta\geq \sum_{i=1}^n\gamma_i+4$.
\end{inparaenum}
\end{defn}

Condition (ii) is satisfied if $L$ is nef. Condition (iv) is equivalent to $-L\cdot K_{\widetilde R}\geq 4$.

 The statement about smoothness and dimension in Theorem \ref{thm:R11} follows from the following more general result, well--known to experts:

\begin{prop} \label{prop:sed}
  Let $S$ be a smooth projective complex surface and $\xi \in {\rm Num}(S)$
  such that $-\xi \cdot K_S>0$.
  Let $\delta$ be an integer  satisfying $0 \leq \delta \leq p_a(\xi)$.

  If $V^\xi_{\delta}(S)$ is nonempty, it is smooth and every component has the expected dimension $-\xi \cdot K_S+p_a(\xi)-\delta-1$.
\end{prop}

\begin{proof}
  Let $X$ be any curve in $V^\xi_{\delta}(S)$ and let $V^\xi(S)$ be the Hilbert scheme defined in the introduction. Since \[ \deg(\N_{X/S})=\xi^2=\xi \cdot(\xi+K_S)- \xi \cdot K_S=2p_a(\xi)-2-\xi \cdot K_S>2p_a(\xi)-2,\]
  the normal bundle $\N_{X/S}$ is nonspecial, whence $V^\xi(S)$ is smooth at $[X]$ of dimension $h^0(\N_{X/S})=-\xi \cdot K_S+p_a(\xi)-1$ (cf., e.g., \cite[\S 4.3]{ser}).

 Let $\varphi:\widetilde{X} \to S$ be the composition of the normalization $\widetilde{X} \to X$ with the inclusion $X \subset S$ and consider the  {\it normal sheaf}  $\N_{\varphi}$ defined by the short exact sequence
\[
  \xymatrix{
    0 \ar[r] & \T_{\widetilde{X}} \ar[r] & \varphi^*\T_S  \ar[r] & \N_{\varphi} 
    \ar[r] & 0.}\]
The tangent space to $V^\xi_{\delta}(S)$ at $[X]$ is isomorphic to $H^0(\widetilde{X},\N_{\varphi})$, and $\N_{\varphi}$ is a line bundle, as $X$ is nodal (cf., e.g., \cite[\S 3.4.3]{ser} or \cite{DS}). Let $g$ be the geometric genus of $X$.
Since $\deg \N_{\varphi} =-X \cdot K_{S}  +2g-2 >2g-2$ by the above sequence, the line bundle
$\N_{\varphi}$ is nonspecial, and
\[ h^0(\N_{\varphi})=-\xi \cdot K_S+g-1=-\xi \cdot K_S+p_a(\xi)-\delta-1=
  \dim (V^\xi(S))-\delta,\]
which is the expected dimension of $V^\xi_{\delta}(S)$. Thus, $V^\xi_{\delta}(S)$ is smooth at $[X]$ and of the expected dimension. 
\end{proof}

By what we said in the introduction, it suffices to prove Theorem \ref{thm:R11} for the highest possible $\delta$, that is, for $\delta=p_a(L)-1$, in which case the Severi variety in question parametrizes nodal curves of geometric genus one. We will prove the theorem by specializing the points $y_1,\ldots,y_n$ as we now explain.

Let $\eta$ be any of the three nonzero $2$-torsion points of $E$.  The map $E \to R$ defined by mapping $e$ to $e + (e \+ \eta)$ realizes $E$ as an unramified double cover of its image curve
\[ T:= \{ e+ (e\+\eta) \; | \; e \in E\}. \]
We have
\begin{equation} \label{eq:T} 
T \sim -K_R+\f_{\eta}-\f_{e_0} \sim 2\s_{e_0}-2\f_{e_0}+\f_{\eta},
\end{equation}
by \cite[(2.10)]{CaCi}. In particular,
\begin{equation} \label{eq:T2} 
  T \not \sim -K_R \; \; \mbox{and} \; \; 2T \sim -2K_R.
\end{equation}

We denote the strict transform of $T$ on $\widetilde{R}$ by the same symbol.
Suppose that $y_1,\ldots,y_n\in T$ are general points. Then
by \eqref{eq:T}--\eqref{eq:T2} we have
\[ 
  T \sim   2\s_{e_0}-2\f_{e_0}+\f_{\eta}-\e_1-\cdots-\e_n \not \sim -K_{\widetilde{R}}, \; \; \; 2T  \sim -2K_{\widetilde{R}}  
\]
on $\widetilde{R}$. In particular,
\[ T \equiv -K_{\widetilde{R}} \equiv 2\s-\f-\e_1-\cdots-\e_n.\]

As remarked in the introduction, 
Theorem \ref{thm:R11} is a consequence of the following result, which
we will prove in \S \ref {sec:proof}.

\begin{prop} \label{prop:R11} Let $E$ be a general irreducible smooth projective curve of genus one.
  Let $y_1,\ldots,y_n \in T$ be general points, with $T$ on $R={\rm Sym}^2(E)$ as defined above. Let $L$ be a line bundle on $\widetilde{R}=\Bl_{y_1,\ldots,y_n}(R)$ verifying condition $(\star)$ with class $\xi$ in ${\rm Num}(\widetilde{R})$.  Then $V_{p_a(L)-1}^\xi(\widetilde{R})$ is nonempty and smooth, of the expected dimension $L \cdot T=-L\cdot K_{\widetilde R}$.
\end{prop}

 \section{A degeneration of the twofold symmetric product of a general elliptic curve} \label{sec:deg}

Let $E$ be a smooth irreducible projective  elliptic curve. 
We recall a degeneration of $R=\Sym^2(E)$ from \cite{ck}, to which we refer the reader for  details.

Let $\X \to \DD$ be a flat projective family of curves over the unit disk $\DD$, with $\X$ smooth, such that the fiber $X_0$ over $0\in \DD$ is an irreducible rational $1$-nodal curve and all other fibers $X_t$, $t \in \DD\setminus \{0\}$, are smooth irreducible elliptic curves. Let $p:\Y \to \DD$ be the relative $2$-symmetric product. Then, for
$t \neq 0$, the fiber $Y_t=p^{-1}(t)\cong \Sym^2(X_t)$ is smooth, whereas the special fiber  $Y_0=p^{-1}(0)= \Sym^2(X_0)$ is irreducible, but   singular.  The singular locus of $Y_0$ consists of the curve
\[ X_P:=\{ x+P \; | \; x \in X_0\}, \] where $P$ is the node of $X_0$.

Let $\nu:\PP^1 \cong \widetilde{X}_0 \to X_0$ be the normalization, with $P_1$ and $P_2$ the preimages of $P$ (with the notation of \cite[p.~328]{ck}, this is the case $g=1$ with $P=P_1$).
Then $\nu$ induces a birational morphism
\begin{equation}\label{eq:sym}
 \Sym^2(\nu): \PP^2 \cong \Sym^2(\widetilde{X}_0) \longrightarrow \Sym^2(X_0)=Y_0.
 \end{equation}
Under the isomorphism on the left the diagonal in $\Sym^2(\widetilde{X}_0)$ corresponds to a smooth conic
$\Gamma$ in $\PP^2$ that is mapped by $\Sym^2(\nu)$ to the diagonal  $\Delta_0$  of $Y_0$ and, for each $x \in \widetilde{X}_0$, the curve
\[ \{x+Q \; | \; Q \in \widetilde{X}_0\} \subset \Sym^2(\widetilde{X}_0)\]
 corresponds to the line in $\PP^2$  tangent to $\Gamma$ at the point corresponding to $2x$.

The threefold $\Y$    has local equations at the point $2P\in Y_0$ given by
$$
z_5^2-z_3 z_4=0,\,\,\,\,z_1 z_5+z_2 z_3=0,\,\,\,\,
z_1 z_4+z_2 z_5=0,\,\,\,\,z_1 z_2 + z_5+t=0,
$$
with $(z_1,...,z_5,t)\in\mathbb C^5\times\DD$ (cf. \cite[p. 329]{ck}). In particular, $\Y$ is singular only at the point $2P$, corresponding to the origin
$\underline 0$. 
Its special fibre $Y_0$ is locally reducible at
$\underline 0=2P$, where it consists of three irreducible components $S^1\cup S^2\cup S^3$ (named $S^i_1$ in \cite{ck}), where $S^1$ is the $z_2z_4$-plane,
$S^2$ is the $z_1z_3$-plane and $S^3$ has equations $z_3=z_1^2,\,z_4=z_2^2, z_5=-z_1z_2$,  meeting as in \cite[Fig. 1]{ck}. 
The singular locus $X_P=\Sing(Y_0)$ of $Y_0$ has a node at the origin, $Y_0$ has double normal crossing singularities
along $X_P\setminus 2P$ and the intersection curves $C^1=S^3\cap S^1$ and $C^2=S^3\cap S^2$ (named $C^i_1$ in \cite{ck}) are the two branches of the curve $X_P$ at $\underline 0=2P$.  Finally, in these local coordinates, the diagonal $\Delta_0$ of $Y_0$ ($\Delta_0=\Delta_{0,1}^1\cup\Delta_{0,2}^1$ in \cite[Fig. 1]{ck}) consists of the $z_2,z_1$-axes and it has a node at the point $2P$. 
 
Let $\mu:\widetilde{\Y} \to \Y$ be the blow--up at the point $2P \in \Sym^2(X_0)=Y_0 \subset \Y$ and denote the exceptional divisor by $\E$ (called $E_1$ in  \cite{ck}). Then $\E \cong \FF_1$ and $\widetilde{\Y}$ is smooth (see \cite[p. 330]{ck}). All fibers over $t \neq 0$ are unchanged.
The special fiber $\widetilde{Y}_0$ of $\widetilde{\Y} \to \DD$ is the union of
$\E$ and of an irreducible surface $\widetilde{S}$, which is the strict transform of $Y_0$. We have
\[ \widetilde{S} \cap \E= s_0+e_1+e_2,\]
where $e_1$ and $e_2$ are two fibers of $\E \cong \FF_1$ and $s_0$ (called $\eta_1$ in \cite[Fig. 2]{ck}) is the section satisfying $s_0^2=-1$. The surface  $\widetilde{S}$ is singular, with  double  normal crossings singularities  along the proper transform
$\widetilde{X}_P$ of the curve $X_P$. The proper transform on $\widetilde{\Y}$ of the diagonal of $\Y$ intersects $\widetilde{Y}_0$ along
\[ \widetilde{\Delta}_0+s_0,\]
where  $\widetilde{\Delta}_0$ is the proper transform of the diagonal
$\Delta_0$ on $Y_0$.

To normalize $\widetilde{S}$ one unfolds along $\widetilde{X}_P$. The resulting surface $W$ is smooth. Denote the normalization map by $\sigma: W \to \widetilde{S}$. The preimage of $\widetilde{X}_P$ is a pair of curves, which we call
$\widetilde{X}_{P_1}$ and $\widetilde{X}_{P_2}$. Denoting the inverse images on $W$ of the curves $e_1,e_2,s_0$ on $\widetilde{S}$ by the same symbols, the intersection configuration between the curves $e_1,e_2,s_0,\widetilde{X}_{P_1},\widetilde{X}_{P_2}$ on $W$ looks as follows:

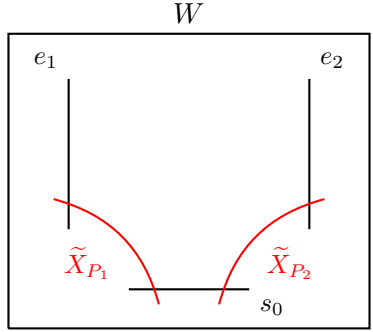
\begin{center}
\begin{tikzpicture}[scale=0.4]

\draw[thick] (-6,-1.3)  -- (6,-1.3) -- (6,8.5) -- node[above] {$W$}(-6,8.5) -- (-6,-1.3);
\draw[thick] (-2,-0)  -- (2,0) node[below right] {\small $s_0$};
\draw[thick] (4,2)  -- (4,7) node[above right] {\small $e_2$};
\draw[thick] (-4,2)  -- (-4,7) node[above left] {\small $e_1$};
\draw[red,thick] (1,-0.5)   to[out=75,in=195]  node[below right] {\small $\widetilde{X}_{P_2}$}(4.5,3);
\draw[red,thick] (-1,-0.5)   to[out=105,in=-15]  node[below left] {\small $\widetilde{X}_{P_1}$}(-4.5,3);

\end{tikzpicture}
\end{center}

Under the map $\sigma$ the two curves $\widetilde{X}_{P_1}$ and $\widetilde{X}_{P_2}$ are identified: we  denote the identification morphism by $\omega:\widetilde{X}_{P_1} \cong \widetilde{X}_{P_2}$. Under this morphism, the intersection points of the above configuration are mapped as follows:

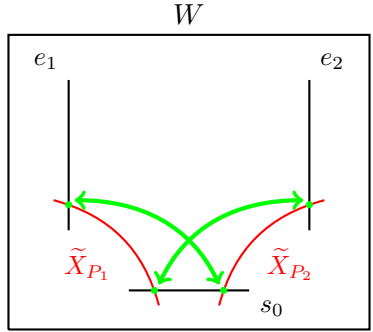
\begin{center}
\begin{tikzpicture}[scale=0.4]

\draw[thick] (-6,-1.3)  -- (6,-1.3) -- (6,8.5) -- node[above] {$W$}(-6,8.5) -- (-6,-1.3);
\draw[thick] (-2,0)  -- (2,0) node[below right] {\small $s_0$};
\draw[thick] (4,2)  -- (4,7) node[above right] {\small $e_2$};
\draw[thick] (-4,2)  -- (-4,7) node[above left] {\small $e_1$};
\draw[red,thick] (1,-0.5)   to[out=75,in=195]  node[below right] {\small $\widetilde{X}_{P_2}$}(4.5,3);
\draw[red,thick] (-1,-0.5)   to[out=105,in=-15]  node[below left] {\small $\widetilde{X}_{P_1}$}(-4.5,3);

\filldraw[green] (4,2.85) circle (0.1);
\filldraw[green] (-1.15,0) circle (0.1);
\draw[green, ultra thick, <->] (-1,0.2)   to[out=65,in=180]  (3.8,3);

\filldraw[green] (-4,2.85) circle (0.1);
\filldraw[green] (1.15,0) circle (0.1);
\draw[green, ultra thick, <->] (1,0.2)   to[out=115,in=0]  (-3.8,3);

\end{tikzpicture}
\end{center}

\begin{defn} \label{def:compa}
  We say that a curve $C \subset W$ is {\it $\omega$-compatible} if $C$ contains 
  neither $\widetilde{X}_{P_1}$ nor $\widetilde{X}_{P_2}$  and 
   $\omega$ maps the $0$--dimensional  intersection scheme of $C$ with $\widetilde{X}_{P_1}$ to the intersection scheme of $C$ with $\widetilde{X}_{P_2}$. 
\end{defn}

If the curve $C$ is $\omega$--compatible then $\sigma(C)$ is a Cartier divisor on $\widetilde{S}$. Conversely any curve on $\widetilde S$ that is a Cartier divisor and does not contain the singular curve of $\widetilde S$  is the image by $\sigma$ of  an $\omega$--compatible curve on $W$.

One sees that the curves $s_0,e_1,e_2$ are $(-1)$--curves on $W$ (see \cite[p. 331--332]{ck}). 
Contracting them we obtain  a morphism $\phi:W \to \PP^2\cong {\rm Sym}^2(\widetilde X_0)$ such that
\[ \phi(s_0)=P_1+P_2, \; \; \phi(e_1)=2P_1, \; \; \phi(e_2)=2P_2\]
and
\[ \phi(\widetilde{X}_{P_i})=X_{P_i}:=\{P_i+Q \; | \; Q \in \widetilde{X}_0\},\]
fitting in a commutative diagram
\[ \xymatrix{
    W \ar[r]^{\sigma} \ar[d]_{\phi} & \widetilde{S}  \ar@{}[r]|-*[@]{\subset} & \hspace{-0.2cm}  \widetilde{S} \cup \E  = \hspace{-0.9cm}
    &   \widetilde{Y}_0  \ar[d]^{\mu|_{\widetilde{Y}_0}} \ar@{}[r]|-*[@]{\subset}  & \ar[d]^{\mu} \widetilde{\Y}  \\
    \PP^2\ar[rrr]_{\Sym^2(\nu)} & & & Y_0 \ar[d]^{p|_{Y_0}}   \ar@{}[r]|-*[@]{\subset}   & \Y \ar[d]^{p} \\
& & &  \{0\}  \ar@{}[r]|-*[@]{\subset}   & \DD 
  }
\] 
(see \cite [p. 332]{ck}). This is shown in the next picture:

\begin{center}
\begin{tikzpicture}[scale=0.4]

\draw[thick] (-6,-1.3)  -- (6,-1.3) -- (6,8.5) -- node[above] {$W$}(-6,8.5) -- (-6,-1.3);
\draw[thick] (-2,-0)  -- (2,0) node[below right] {\small $s_0$};
\draw[thick] (4,2)  -- (4,7) node[above right] {\small $e_2$};
\draw[thick] (-4,2)  -- (-4,7) node[above left] {\small $e_1$};
\draw[red,thick] (1,-0.5)   to[out=75,in=195]  node[below right] {\small $\widetilde{X}_{P_2}$}(4.5,3);
\draw[red,thick] (-1,-0.5)   to[out=105,in=-15]  node[below left] {\small $\widetilde{X}_{P_1}$}(-4.5,3);

\draw[thick,->] (8,3) -- node[above]{$\phi$}(10,3);

\draw[thick] (12,-1.3)  -- (24,-1.3) -- (24,8.5) -- node[above]
{$\mathbb{P}^2$}(12,8.5) -- (12,-1.3);
\draw[red,thick] (17,-.5)  -- (22.5,4);
\draw[red,thick] (19,-.5) -- (13.5,4);

\filldraw (18,.3) circle (0.1) node [right] {\small $P_1+P_2$};
\filldraw (21.9,3.5) circle (0.1) node [right] {\small $2P_2$};
\filldraw (14.1,3.5) circle (0.1) node [left] {\small $2P_1$};
\end{tikzpicture}

\captionof{figure}{}
\label{fig:WP}
\end{center}

\begin{remark}\label{rem:conic} The morphism $\omega: \widetilde{X}_{P_1} \to \widetilde{X}_{P_2}$ is geometrically interpreted in the following way (see \cite {ck}). Via $\phi$ the curves $\widetilde{X}_{P_1}$ and $\widetilde{X}_{P_2}$ map  isomorphically  to the two lines on the plane $\PP^2$ in red in Figure~\ref{fig:WP} joining the point $P_1+P_2$ with the points $2P_1$ and $2P_2$ respectively. In $\PP^2$ we have the conic $\Gamma$  (mapped by $\Sym^2(\nu)$ to the diagonal $\Delta_0$ of $Y_0$), which   is tangent to these lines at the points $2P_1$ and $2P_2$. The map $\omega$ associates two points if and only if their images in the plane lie on a tangent line to $\Gamma$. 
  (The two points $P_1+Q$ and $P_2+Q$ of $\mathbb P^2$  lie on the tangent line to $\Gamma$ at the point $2Q$ and are the intersection points of this tangent line with the two lines joining $2P_1$ with $P_1+P_2$ and $2P_2$ with $P_1+P_2$,
  namely $\phi(\widetilde{X}_{P_1})$ and $\phi(\widetilde{X}_{P_2})$.) 
\end{remark}

The Picard group of $W$ is generated by $h,e_1,e_2,s_0$, where $h$ is the pullback by $\phi$ of a line. In particular,
\begin{equation} \label{eq:XP}
  \widetilde{X}_{P_i} \sim h-s_0-e_i, \; \; i=1,2.
  \end{equation}
One has
\[ -K_{W}=3h-e_1-e_2-s_0.\]

Let us look at what happens to the classes of  $\s$ and  $\f$ under the degeneration of $R$ to $\widetilde{Y}_0$. This is described in \cite [\S 2]{ck} together with the more general description of the degeneration of line bundles on $R$ under the degeneration of $R$ to $\widetilde{Y}_0$, which we are now going to recall.  

Let $h'$ be an $\omega$-compatible member of $|h|$ on $W$ (cf. Definition \ref{def:compa}), not containing any of $e_1$, $e_2$ or $s_0$. There is a one-dimensional irreducible  family of such curves whose general member is the strict transform on $W$ of a general tangent line to the conic $\Gamma$ of $\PP^2$ mapped to the diagonal of $Y_0$ by $\Sym^2(\nu)$  (cf. Remark \ref{rem:conic}).  Since $h \cdot e_1=h \cdot e_2=h \cdot s_0=0$, we have $\sigma(h') \cap \E=\emptyset$, so that $\sigma(h')\subset \widetilde S$ determines a Cartier divisor on $\widetilde{Y}_0$. The class of $\s$ on $R$ degenerates to the class of $\sigma(h')$.

The class $h-s_0$ on $W$ satisfies $(h-s_0) \cdot s_0=1$ and $(h-s_0) \cdot e_i=(h-s_0) \cdot \widetilde{X}_{P_i}=0$, $i=1,2$. Thus the general member $F$ of the pencil $|h-s_0|$ is $\omega$-compatible and $\sigma(F)$ intersects $\E$ in one point along $s_0$. Therefore, the union of $\sigma(F)$
with the fiber of $\E$ over the intersection point on $s_0$ is a Cartier divisor on $\widetilde{Y}_0$, which turns out to be the limit of $\f$.

Let $C \equiv a\s+b\f$ on $R$, with $a,b \geq 0$, and let $C_0$ be its limit on
$\widetilde{Y}_0$. Assume that it neither contains  any of $e_1, e_2, s_0$ nor the double curve of $\widetilde S$. We may write $C_0=C_{\widetilde{S}} \cup C_{\E}$ with $C_{\widetilde{S}} \subset \widetilde{S}$ and
$C_{\E} \subset \E$. Then $C_{\widetilde{S}} \cap C_{\E} \subset s_0$ and $C_{\E}$ is a union of fibers of $\E$. We have
$C_{\widetilde{S}}=\sigma(C_W)$, with $C_W$ a $\omega$-compatible curve satisfying
\[ C_W \sim ah+b(h-s_0) =(a+b)h-bs_0.\]
This is because the transform of the limit of $\s$ is numerically equivalent to $h$ on $W$ and the transform of the limit of $\f$ is equivalent to $(h-s_0)$.
This means that $\phi(C_W) \subset \PP^2$ is a plane curve of degree $a+b$ with a point of multiplicity $b$ at $P_1+P_2$, with intersection points with $X_{P_1}$ and $X_{P_2}$ satisfying the suitable ``gluing conditions'' given by $\omega$. 

Conversely, we have: 

\begin{lemma} \label{lemma:w1}
  Let $a,b \geq 0$ and $C_W \in |(a+b)h-bs_0|$ be  an $\omega$-compatible curve not containing any of $e_1,e_2,s_0$ and intersecting $s_0$ in distinct points. Let $C_{\E}$ denote the union of fibers on $\E \cong \FF_1$ such that $C_{\E} \cap s_0 =C_W \cap s_0$. 
Then $\sigma(C_W) \cup C_{\E}$  is the flat limit of a curve algebraically equivalent to $a\s+b\f$. 
\end{lemma}

\begin{proof}
Since   
$C_W \cdot \widetilde{X}_{P_i}=a$, the locus of $\omega$-compatible curves in $|C_W|$ has dimension 
\[ \dim |C_W|-a=\frac{1}{2}\left(a^2+3a+2ab+2b\right)-a=\frac{1}{2}\left(a^2+a+2ab+2b\right),\]
which equals the dimension of the Hilbert scheme of curves algebraically equivalent to $a\s+b\f$. The result follows from the discussion prior to the lemma.
\end{proof}

Let us now go back to the degeneration $\X \to \DD$ of  a general elliptic curve $E$ to a rational nodal curve $X_0$ we considered at the beginning of this section. This can be viewed as a degeneration of the group $E$ to $\mathbb C^*$, where (keeping the notation introduced at the beginning of this section) $\mathbb C^*=\PP^1\setminus \{P_1,P_2\}$. Since $\mathbb C^*$ has a unique non--trivial point of order 2, i.e., $-1$, we see that in the degeneration $\X \to \DD$ one of the three non--trivial points of order 2 of the general fibre degenerates to $-1$, so it is fixed by the monodromy of the family $\X \to \DD$. (The other two non--trivial points of order 2 must degenerate to the node of $X_0$.) 
This implies that  we have a divisor $\T$ on $\widetilde{\Y}$ such that the fiber $T_t$ for $t \neq 0$ is a curve $T \subset R$ like the one we considered in \S \ref {sec:prel}. Since $T\equiv 2\mathfrak s-\f$, the proper transform $T_W$ on $W$ of the limit of the curve $T$ is such that 
$$T_W \sim 2h-(h-s_0)\sim h+s_0$$
(remember that the pull--back on $W$ of the limit of $\s$ and $\f$ are $h$ and $h-s_0$ respectively). More precisely, since $T$ has zero intersection with the diagonal of $R$, the image of $T_W$  in $\PP^2$ via $\phi: W\to \PP^2$ must intersect the conic $\Gamma$,  which is mapped by $\Sym^2(\nu)$ to the diagonal of $Y_0$ (see \eqref {eq:sym}), only in points of $\Gamma$ that are blown up by $\phi$, i.e., in the points corresponding to $2P_1$ and $2P_2$ (see Figure~\ref{fig:WP}). This implies that 
$$
  T_W=h_{0}+e_1+e_2+s_0,
$$
where $h_0$ is the strict transform by $\phi$ of the line in $\PP^2$ through $2P_1$ and $2P_2$. 

Since $T_W$ contains $s_0$, the divisor $\T$ contains $\mathcal E$. By subtracting $\mathcal E$ from $\T$, the resulting irreducible effective divisor $\T-\mathcal E$ intersects the central fibre in a curve that consists of two components: one component on $\widetilde S$,  which is $\sigma(h_0)$, 
and another component sitting on $\mathcal E$ that is the pull--back on $\mathcal E\cong \mathbb F_1$ of the unique line of $\PP^2$ passing through the two points in which $h_0$ intersects $e_1$ and $e_2$. However, what will be important for us in what follows is that $\sigma(h_0)$    is in the limit of $T$.

\section{A useful family of rational curves on some blow--ups of the plane}\label{sec:lemmas}

In this section we prove some results on certain line bundles on some blow--ups of the surface $W$ introduced in the previous section. They will be useful in the proof of Proposition \ref {prop:R11} in \S \ref{sec:proof}. We go on keeping the notation and convention we introduced in the previous section.

Let $y_1,\ldots,y_n \in h_0$ be general points. Choose sections of  $p:\widetilde{\Y} \to \DD$ passing through $\sigma(y_1),\ldots,\sigma(y_n) \in \sigma(h_0)$  and through general points $y^t_1,\ldots,y^t_n \in T_t$ on a general fiber $Y_t$. 
Blowing up $\widetilde\Y$ along these sections, we obtain a smooth threefold $\Y'$ with a morphism  $p':\Y' \to \DD$ with general fiber the blow--up of $Y_t=\Sym^2(X_t)$ at $n$ general points of $T_t$ and 
special fiber $Y':=S' \cup \E$, where $S'=\Bl_{\sigma(y_1),\ldots,\sigma(y_n)} (\widetilde{S})$, and there is a normalization morphism $\sigma':W' \to S'$, where $W'=\Bl_{y_1,\ldots,y_n}(W)$. Let $e_{y_i}$ denote the  exceptional divisor in $W'$ over $y_i$, for $i=1,\ldots, n$.
We denote the strict transforms of $e_1,e_2,s_0,\widetilde{X}_{P_1},\widetilde{X}_{P_2},h_0$ on $W'$ by the same symbols. Note that \eqref{eq:XP} still holds; furthermore 
$$
 h_0 \sim h-e_1-e_2-e_{y_1}-\cdots-e_{y_n}
$$
and
 \begin{equation}
   \label{eq:kw}
 -K_{W'}=3h-e_1-e_2-s_0-e_{y_1}-\cdots-e_{y_n} \sim h_0+e_1+e_2+\widetilde{X}_{P_1}+\widetilde{X}_{P_2}+s_0.
 \end{equation}
Moreover, the pull--back on $W'$ of the limit of $T$ on $S'$ contains $h_0$. 

We next fix a general point $x_1 \in \widetilde{X}_{P_1}$ and set $x_2=\omega(x_2) \in  \widetilde{X}_{P_2}$. The following picture summarizes the situation:

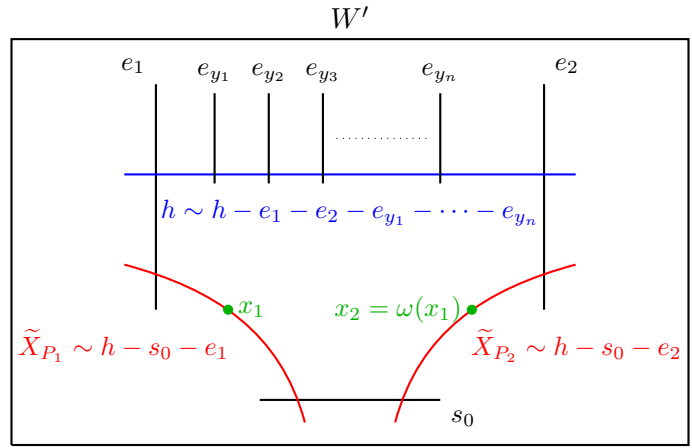
\begin{center}
\begin{tikzpicture}[scale=0.6]

\draw[thick] (-7.5,-1)  -- (7.5,-1) -- (7.5,8) -- node[above] {$W'$}(-7.5,8) -- (-7.5,-1);
\draw[thick] (-2,-0)  -- (2,0) node[below right] {\small $s_0$};
\draw[thick] (4.3,2)  -- (4.3,7) node[above right] {\small $e_2$};
\draw[thick] (-4.3,2)  -- (-4.3,7) node[above left] {\small $e_1$};
\draw[red,thick] (1,-0.5)   to[out=75,in=195]  node[below right] {\small $\widetilde{X}_{P_2}\sim h-s_0-e_2$}(5,3);
\draw[red,thick] (-1,-0.5)   to[out=105,in=-15]  node[below left] {\small $\widetilde{X}_{P_1} \sim h-s_0-e_1$}(-5,3);

\draw[blue,thick] (-5,5)   -- node[below=0.2cm] {\small $h \sim h-e_1-e_2-e_{y_1}-\cdots -e_{y_n}$}(5,5);

\draw[thick] (-3,4.8)  -- (-3,6.8) node[above] {\small $e_{y_1}$};
\draw[thick] (-1.8,4.8)  -- (-1.8,6.8) node[above] {\small $e_{y_2}$};
\draw[thick] (-0.6,4.8)  -- (-0.6,6.8) node[above] {\small $e_{y_3}$};
\draw[dotted] (-0.3,5.8)  -- (1.7,5.8);
\draw[thick] (2,4.8)  -- (2,6.8) node[above] {\small $e_{y_n}$};

\filldraw[black!30!green] (2.7,2) circle (0.1) node [left] {\small $x_2=\omega(x_1)$};
\filldraw[black!30!green] (-2.7,2) circle (0.1) node [right] {\small $x_1$};

\end{tikzpicture}
\end{center}

\renewcommand{\M}M
We introduce the following notation. 
 For a line bundle $\M$ on $W'$, we denote by $V_{\M}$ the locus of curves $C$ in $|\M|$ on $W'$ such that
\begin{itemize}
  \item $C$ is irreducible and rational,
  \item $C$ intersects $\widetilde{X}_{P_i}$ only at $x_i$, $i=1,2$, and it is unibranch there.
  \end{itemize}
We denote by $V_{\M}^*$ the open sublocus of $V_{\M}$ of curves $C$ with the further property that
\begin{itemize}
  \item $C$ intersects $s_0$ transversely,
  \item $C$ is smooth at $x_i$, $i=1,2$,
  \item $C$ is nodal.
  \end{itemize}

  \begin{lemma}\label{lemma:TD}
    Assume $V_{\M} \neq \emptyset$.

    (i) If $\M \cdot(h+s_0-e_{y_1}-\cdots-e_{y_n}) \geq 1$, then for each component $V$ of $V_{\M}$ one has 
\[
    \dim (V)=-K_{W'}\cdot \M-1-\M \cdot \widetilde{X}_{P_1}-\M \cdot \widetilde{X}_{P_2}=
    \M \cdot(h+s_0-e_{y_1}-\cdots-e_{y_n})-1. 
 \]
   
(ii) If  $\M \cdot(h+s_0-e_{y_1}-\cdots-e_{y_n}) \geq 4$, then 
    $V_{\M}^*\neq \emptyset$. 
  \end{lemma}
  
\begin{proof}  The result follows from  \cite[\S 2]{CH}, as outlined in \cite[Thm. (1.4)]{noteTD}.\end{proof} 

\begin{prop}\label{prop:strafico} 
Let $\alpha, \beta, \gamma_1,\ldots, \gamma_n$ be non--negative integers verifying conditions (i)--(iv) in Definition \ref {def:good}. Set $\M=(\alpha+\beta)h-\beta s_0-\sum \gamma_i e_{y_i}$. Then $V^*_\M$ is non--empty with all components of dimension $\alpha+2\beta-\sum \gamma_i-1$.
\end{prop}

In the proof of Proposition \ref {prop:strafico} we will need the following:

\begin{lemma}\label{lemma:checazzo!} Given three lines $\ell_1,\ell_2,\ell_3$ in the plane $\PP$ not passing through the same point, we  set $y_{ij}=\ell_i\cap \ell_j$ for $1 \leq i < j \leq 3$.  Fix integers $d> m\geq 0$, $n\geq 0$, $m_1,\ldots, m_n\geq 1$, such that 
$$
d\geq \sum_{i=1}^n m_i \quad \text{and}\quad d\geq m+m_i, \quad i=1,\ldots, n.
$$
Then there is a reduced and irreducible rational curve $\gamma$ in $\PP$ of degree $d$ with the following properties:
\begin{itemize}
\item $\gamma$ has a  point of multiplicity $m$ at $y_{12}$,
\item the pull--back on the normalization of $\gamma$ of the $g^1_{d-m}$ cut out by the lines through $y_{12}$ has two total ramification points  mapping to generic points $x_1 \in \ell_1$ and $x_2 \in \ell_2$ respectively (in particular different from $y_{12},y_{13},y_{23}$),
\item  $\gamma$ has $n$ points of multiplicities $m_1,\ldots, m_n$ that are pairwise distinct points
  on $\ell_3$ (in particular different from $y_{13}$ and $y_{23}$).
\end{itemize}
\end{lemma}

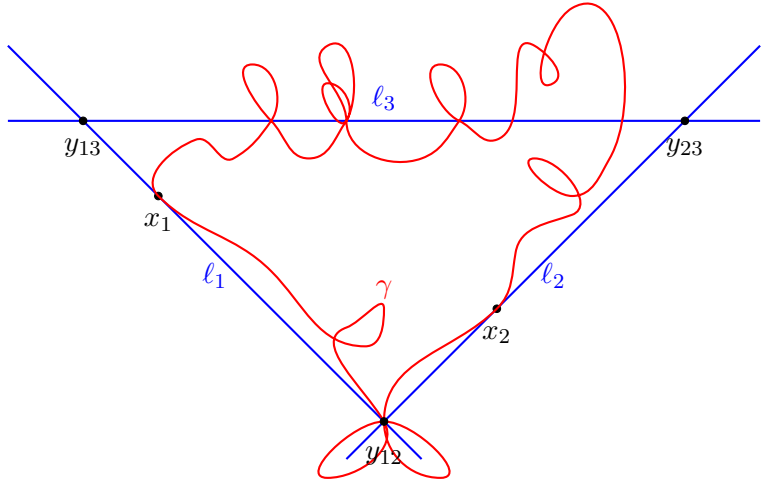
\begin{center}
\begin{tikzpicture}[scale=0.5]

\draw[blue,thick] (-10,10)  -- node[below] {$\ell_1$} (1,-1);
\draw[blue,thick] (-1,-1)  -- node[below] {$\ell_2$} (10,10);
\draw[blue,thick] (-10,8) -- node[above] {$\ell_3$} (10,8);

\filldraw (-8,8) circle (0.1) node [below=0.1cm,color=black] {$y_{13}$};
\filldraw (8,8) circle (0.1) node [below=0.1cm,color=black] {$y_{23}$};

\filldraw (-6,6) circle (0.1) node [below=0.1cm,color=black] {$x_1$};
\filldraw (3,3) circle (0.1) node [below=0.1cm,color=black] {$x_{2}$};

\draw[smooth,red,thick]
(-4,7) to[out=30,in=-125]
(-3,8) to[out=55,in=0] (-3.5,9.5) to[out=180,in=135] (-3,8) to[out=-45,in=180]
(-2,7) to[out=0,in=-125]
(-1,8) to[out=55,in=0]  (-1.5,9) to[out=180,in=240] (-1,8)
       to[out=60,in=30] (-1.5,10) to[out=210,in=90] (-1,8)
       to[out=-90,in=200]
(1,7) to[out=20,in=-125]   (2,8) to[out=55,in=0] (1.5,9.5) to[out=180,in=155] (2,8) to[out=-25,in=200]
(3,7.5) to[out=20,in=150] (4,10) to[out=-30,in=30] (4.5,9) to[out=210,in=-150]
(5,11) to[out=30,in=60] (6,7) to[out=-120,in=0] (5,6) to[out=180,in=180] (4,7) 
to[out=0,in=20] (5,5.5) to[out=-160,in=80] (3.5,4.5) to[out=-100,in=45] (3,3)
to[out=225,in=90] (0,0) to[out=-90,in=180] (1.5,-1.5) to[out=0,in=0] (0,0)
to[out=180,in=180] (-1.5,-1.5) 
to[out=0,in=-60] (0,0) 
to[out=120,in=200] (-1,2.5) 
to[out=20,in=90] (0,3)
to[out=-90,in=0] (-0.5,2)
to[out=180,in=-45] (-3,4)
to[out=135,in=-45] (-6,6)
to[out=135,in=200] (-5,7.5)
to[out=20,in=210] (-4,7)
;

\draw[red] (0,3.5) node {$\gamma$};

\filldraw (0,0) circle (0.1) node [below=0.2cm,color=black] {$y_{12}$};

\end{tikzpicture}

\end{center} 

\begin{proof} Set $\delta=d-m$. The assertion is trivial if $\delta=1$.  So we assume $\delta\geq 2$. Consider a morphism $f: \PP^1\to \PP^1$ of degree $\delta$ with two  points of total ramification, that is a $g^1_\delta$ with no base points. Fix a general effective divisor $D$ of degree $m$ on $\PP^1$, so that $g^1_\delta+D$ is a $g^1_d$. Fix $n$ general points $P_1,\ldots, P_n$ of $\PP^1$, and consider the fibres
$$
f^{-1}(P_i)=P_{i1}+\cdots +P_{i\delta}, \quad i=1,\ldots, n.
$$
Consider then the divisor
$$
E=\sum_{i=1}^n \sum_{j=1}^{m_i} P_{ij}+ F
$$
where $F$ is a general effective divisor of degree $d-\sum_{i=1}^nm_i$ on $\PP^1$. The divisor $E$  has no common point with the general divisor of $g^1_\delta+D$. Hence $E$ and $g^1_\delta+D$ span a $g^2_d$ with no base points. Moreover, this $g^2_d$ is birational. Indeed, if $g^2_d$ were composed with a $g^1_\nu$, then, by the generality of the divisor $D$,  the $g^1_\delta$ would have base points, a contradiction. 

Let $\gamma$ be the image of $\PP^1$ via the $g^2_d$. This is a rational plane curve of degree $d$, with a point  of multiplicity $m$ at a point $y_{12}$ of $\PP^2$. Moreover there are two lines $\ell_1,\ell_2$ passing through $y_{12}$, that each intersects $\gamma$ at one point apart from $y_{12}$, call it 
$x_1$ and $x_2$ respectively, where the $g^1_\delta$ has total ramification. Finally there is a third line $\ell_3$ that pulls back to $\PP^1$ to the divisor $E$. By the choices we made, this line does not pass through $x_1$ and $x_2$ and the divisors $\sum_{j=1}^{m_i} P_{ij}$, for $i=1,\ldots,n$, are contracted by the $g^2_d$ to $n$ distinct points on $\ell_3$ that have multiplicities $m_1,\ldots, m_n$. The genericity of $x_1,x_2$ can be achieved by acting with projective transformations of the plane fixing the lines $\ell_1,\ell_2,\ell_3$, which keep the points of multiplicities $m_1,\ldots, m_n$ pairwise distinct.
\end{proof}

\begin{proof}[Proof of Proposition \ref  {prop:strafico}]  Let $d=\alpha+\beta$, $m_i=\gamma_i$ and $m=\beta$. Consider the plane $\PP$ containing the curve $\gamma$ constructed in Lemma \ref {lemma:checazzo!}. Let us blow up the points $y_{12},y_{13},y_{23}$ and the $n$ points of multiplicities $m_1,\ldots, m_n$ on $\gamma$ along $\ell_3$. 

We will consider the family $\mathcal W$ of the surfaces $W'$ as above where the points $y_1,\ldots, y_n$ are no longer general but simply pairwise distinct. We call $b>0$ the dimension of the parameter space of this family. There is a line bundle $\mathcal M$ on $\mathcal W$ that restricts on each member of $\mathcal W$ to the line bundle $M$ as in the statement of the proposition. Accordingly we can consider the families $\mathcal V_{\mathcal M}$ and $\mathcal V_{\mathcal M}^*$ of all varieties  $V_{M}$ and $V_{M}^*$ as before.  

The blow--up at the beginning of the proof can be interpreted as a member $W'_0$ of $\mathcal W$ with $\omega(x_1)=x_2$, since there is a unique irreducible conic $\Gamma$ tangent to the lines $\ell_1,\ell_2$ at $y_{13}$ and $y_{23}$ respectively, and tangent also to the line joining the two points $x_1$ and $x_2$ (cf. Remark \ref {rem:conic}). We denote by $M_0$ the restriction of $\mathcal M$ to $W'_0$. 

 Lemma  \ref {lemma:checazzo!} implies that $V_{M_0}$ is non--empty, which by Lemma \ref {lemma:TD} implies in turn that  $V^*_{M_0}$ is non--empty, with all components of  the expected dimension $\alpha+2\beta-\sum \gamma_i-1$. This yields that $\mathcal V_{\mathcal M}^*$ is non--empty. Then its dimension is at least the expected dimension which is $\alpha+2\beta-\sum \gamma_i-1+b$ that is strictly larger than the dimension  
of  $V^*_{M_0}$. This implies that for $W'$ general in $\mathcal W$ the variety $V_{M}^*$ is non--empty of the expected dimension $\alpha+2\beta-\sum \gamma_i-1$, as wanted. \end{proof}

  \section{Proof of Proposition \ref {prop:R11}}\label{sec:proof}

Let us go back to $R$ and $\widetilde R=\Bl_{y_1,\ldots,y_n}(R)$, where $y_1,\ldots,y_n$ are general points on $T$, with exceptional divisors $\e_i$ over $y_i$, for $i=1,\ldots,n$. 

\begin{proof}[Proof of Proposition \ref{prop:R11}] As we already noted, condition (iv) in Definition \ref {def:good} of $(\star)$ is equivalent to $-K_{\widetilde R}\cdot L\geq 4$. Hence,  as remarked  for Theorem 
\ref {thm:R11} in the introduction, 
the statements about dimension and smoothness follow from  Proposition \ref{prop:sed}  once nonemptiness is proved. 
So it remains to prove nonemptiness. 

We prove the result by degeneration of $\widetilde R$ to $Y'$, as indicated in \S \ref {sec:lemmas}, from which we keep the notation. 

 On the surface $W^\prime$ consider 
\[ L_0\sim \alpha h+ \beta(h-s_0) - \sum \gamma_i e_{y_i}.\]
 
Denote by $\{L_0\}_{W'} \subset |L_0|$ the sublocus of $\omega$--compatible curves.

\begin{claim} \label{cl:ugua}
  $\dim (\{L_0\}_{W'})=\dim (|L_0|)-\alpha=
\frac{1}{2}\left(\alpha^2+\alpha-\sum \gamma_i-\sum \gamma_i^2\right)+\beta(\alpha+1)$.
\end{claim}

\begin{proof}[Proof of claim]
  By Proposition \ref{prop:strafico}, the linear system $|L_0|$ contains an irreducible curve. Since $h^1(\O_{W'})=0$, it therefore follows that
  \begin{equation}
    \label{eq:KWvan}
    h^1(-L_0)=h^1(L_0+K_{W'})=0.
  \end{equation}
Set $A:=h_0+e_1+e_2+\widetilde{X}_{P_1}+\widetilde{X}_{P_2}+s_0$. Then $A$ is a reduced cycle of rational curves, thus of arithmetic genus one, and it is anticanonical by \eqref{eq:kw}. Since $L_0 \cdot A \geq 4$ by condition (iv) of $(\star)$, it follows in particular that $|L_0+A|$ contains a reduced, connected member. It therefore follows that
 \begin{equation}
    \label{eq:KWvan2}
    h^1(-L_0-A)=h^1(L_0)=0 \; \; \mbox{and} \; \; h^0(-L_0-A)=h^2(L_0)=0.
  \end{equation} 
  In particular, using Riemann-Roch on $W'$, one computes that
  \[ \dim (|L_0|)=\frac{1}{2}\left(\alpha^2+3\alpha-\sum \gamma_i-\sum \gamma_i^2\right)+\beta(\alpha+1),\]
  thus proving the 
  right hand equality of the claim.

  We have left to prove that $\{L_0\}_{W'}$ has codimension $\alpha$ in $|L_0|$.

To this end, let $Z_1 \in \Sym^{\alpha}(\widetilde{X}_{P_1})$ be general, and set $Z_2:=\omega(Z_1) \in \Sym^{\alpha}(\widetilde{X}_{P_2})$. 
From the two restriction sequences
  \[
  \xymatrix{
    0 \ar[r] & L_0+K_{W'} \ar[r] & L_0  \ar[r] & L_0|_A \ar[r] & 0}\]
and
\[
  \xymatrix{
    0 \ar[r] & L_0+K_{W'} \ar[r] & L_0 \* \I_{Z_1 \cup Z_2}  \ar[r] & L_0|_A(-Z_1-Z_2) \ar[r] & 0},\]
together with \eqref{eq:KWvan}, we see that
\[ \codim\Big(\left\vert L_0 \* \I_{Z_1 \cup Z_2}\right\vert,\left\vert L_0\right\vert|\Big) = \codim\Big(\left\vert L_0|_A(-Z_1-Z_2)\right\vert,\left\vert L_0\right\vert\Big).\]
A standard computation involving restriction sequences to the various components of $A$ shows that the latter codimension is $2\alpha$. Therefore,
\[ \dim\left(\left\vert L_0 \* \I_{Z_1 \cup Z_2}\right\vert\right)=\dim (|L_0|)-2\alpha.\]
(This equality can also be obtained applying \cite[Thm. (1.4.0)]{noteTD}.)
Varying $Z_1 \in \Sym^{\alpha}(\widetilde{X}_{P_1})$, we obtain the whole of $\{L_0\}_{W'}$. Thus,
\[ \dim (\{L_0\}_{W'})= \dim\left(\left\vert L_0 \* \I_{Z_1 \cup Z_2}\right\vert\right)+\dim(\Sym^{\alpha}(\widetilde{X}_{P_1}))=
  \dim(|L_0|)-\alpha,\]
finishing the proof of the claim.
\end{proof}

Denote
by $\{L_0\}$ the locus of curves on $Y'=S' \cup \E$ of the form
$\sigma'(C) \cup  C_{\E}$, where $C$ is an element of $\{L_0\}_{W'}$ and $C_{\E}$ is  the union of fibers on $\E\cong \FF_1$ such that 
$C_{\E} \cap S'=\sigma^\prime (C)  \cap s_0$. Then there is a one-to-one correspondence between $\{L_0\}_{W'}$ and $\{L_0\}$. Thus,  by the last claim, 
\begin{equation} \label{eq:diml0}
  \dim (\{L_0\})=
\frac{1}{2}\left(\alpha^2+\alpha-\sum \gamma_i-\sum \gamma_i^2\right)+\beta(\alpha+1).
\end{equation}
Note that all members of $\{L_0\}$ are Cartier divisors on $Y'$.
Moreover, by Lemma \ref{lemma:w1}, the closure of the locus $\{L_0\}$ is (a component of) the limit of the algebraic system $\{L\}$ on $\widetilde{R}$ of curves of class $\alpha \s  +\beta \f -\sum \gamma_i \e_i$.  Since the anticanonical divisor on $\widetilde{R}$ is effective, we have $h^2(L)=h^0(K_{\widetilde{R}}-L)=0$, whence Riemann-Roch yields
\begin{eqnarray*}
  \dim \{L\} & = & \dim |L|+1=\chi(L)+h^1(L)
  \\
             & = &\frac{1}{2}L\cdot (L-K_{\widetilde{R}})+h^1(L) \\
   & = &\frac{1}{2}\left(\alpha^2+\alpha-\sum \gamma_i-\sum \gamma_i^2\right)+\beta(\alpha+1)+h^1(L).
                    \end{eqnarray*}
    By \eqref{eq:diml0} and semicontinuity, we must have $h^1(L)=0$ and 
\begin{equation} \label{eq:diml00}
  \dim (\{L_0\})=\dim (\{L\}).
  \end{equation}

Let $x_1$ and $x_2$ be as in \S \ref {sec:lemmas} and
pick a general $C$ in a component of $V^*_{L_0}$ in $W'$ (which is non--empty by Proposition \ref {prop:strafico}).  Then $C$ intersects $s_0$ transversely at $L_0 \cdot s_0=\beta$ distinct points.  Denote as above by $C_{\E}$ the union of the $\beta $ fibers on $\E$ such that $C_{\E} \cap s_0 =\sigma^\prime (C) \cap s_0$. 
Then 
$\sigma'(C) \cup  C_{\E}$  is a member of  $\{L_0\}$,  with an $\alpha$-tacnode at $\sigma'(x_1)=\sigma'(x_2)$, and nodal otherwise, stably equivalent to $\sigma'(C)$. Varying $x_1$, we obtain by Proposition \ref {prop:strafico} a family $\mathcal C$ of dimension
$\alpha+2\beta-\sum \gamma_i=-L\cdot K_{\widetilde R}$
of such curves, and this is the expected dimension of $V_{p_a(L)-1}^\xi(\widetilde{R})$.

Let $\delta_0$ be the number of nodes of $C$, which equals the number of 
singular points of 
$\sigma'(C)$  
on the smooth locus of $Y'$.
Then
\[
  \delta_0=p_a(L_0)=\frac {1}{2}\left(\alpha^2-3\alpha+\sum \gamma_i-\sum \gamma^2_i\right)+\beta(\alpha-1)+1.
\]

Grant for the moment the following\footnote{ From a deformation-theoretic point of view the claim asserts the smoothness of the {\it equisingular deformation locus} of $\sigma'(C) \cup  C_{\E}$, which is a dense open subset of $\C$, cf. \cite[Lemma 3.4]{gk}.}:
  \begin{claim} \label{cl:lisc}
    The family of curves in $\{L_0\}$ passing through the $\delta_0$ nodes of $\sigma'(C)$ and having an $(\alpha-1)$--tacnode at $\sigma'(x_1)=\sigma'(x_2)$ has codimension $\delta_0+\alpha -1$ in $\{L_0\}$, that is, it has dimension
$\alpha+2\beta-\sum \gamma_i=\dim \C$.
\end{claim}

Arguing 
as in \cite[Thm. 3.3,  Cor. 3.12  and proof of Thm. 1.1]{gk}\footnote{The setting in \cite{gk} is slightly different, as the central fibre in the degeneration is a transversal union of two smooth surfaces, whereas $S'$ in the present setting is singular. Moreover, the central fiber in the degeneration in \cite{gk} is regular, whence linear and algebraic equivalence coincide, which is not the case on $Y'$. However, the reasoning in \cite{gk} is local, so the proof goes through in the present setting as well. The two hypotheses (1)-(2) in
  \cite[Thm. 3.3]{gk} correspond, respectively, to  \eqref{eq:diml00} and the statement in Claim \ref{cl:lisc}.}, we may deform 
$Y'$ to $\widetilde R$ deforming the $\alpha$-tacnode  of $\sigma^\prime (C)$  to
$\alpha-1$ nodes, while preserving the $\delta_0$ nodes  and smoothing the nodes $\sigma'(C) \cap C_{\E}$. 
Thus  $\sigma'(C) \cup C_{\E}$  deforms to a curve algebraically equivalent to $L$ with $\delta$ nodes, where
$$
\delta = \delta_0+\alpha-1=\frac 12 \left(\alpha^2-\alpha+\sum  \gamma_i-\sum \gamma^2_i\right)+\beta(\alpha-1).
$$
One computes
$$
p_a(L)=\frac{1}{2}\left(L^2+L \cdot K_{\widetilde R}\right)+1=\delta+1.
$$
  This shows that $C$ has geometric genus one, as wanted.

  We have left to prove the claim.

  \begin{proof}[Proof of Claim \ref{cl:lisc}]
Let $\F$ be the family of curves in $\{L_0\}_{W'}$ passing through the $\delta_0$ nodes of $C$ and
being tangent to $\widetilde{X}_{P_i}$ at $x_i$ with order $\alpha-1$, for $i=1,2$.  The statement of the claim is equivalent to $\dim \F=\alpha+2\beta-\sum \gamma_i$.

Denoting by $N$ the scheme of the $\delta_0$ nodes of $C$ and by $Z_i=(\alpha-1)x_i$ the subscheme on $\widetilde{X}_i$, whence on $W'$, we have that $\F$ is the locus of $\omega$-compatible curves in
$|L_0 \* \I_{N \cup Z_1 \cup Z_2}|$, which has codimension one, as $L_0 \cdot \widetilde{X}_i=\alpha$. Thus,
\[
  \dim(\F)=\dim (|L_0 \* \I_{N \cup Z_1 \cup Z_2}|)-1.
\]
To compute this, let $q: W'' \to W'$ denote the blow--up of $W'$ at $N$ and denote the total exceptional divisor by $E$. Denote the inverse images of $Z_i$ by the same names. Then
\begin{equation}
  \label{eq:dimsubloc2}
  \dim(\F)=\dim \left(|(q^*L_0-E)  \* \I_{Z_1 \cup Z_2}|\right)-1.
\end{equation}
Let $\widetilde{C}$ be the strict transform of $C$ on $W''$, which is a smooth rational curve. Then $\widetilde{C} \sim q^*L_0-2E$. We therefore have a short exact sequence
\[
    0 \longrightarrow  \O_{W''}(E) \longrightarrow (q^*L_0-E)  \* \I_{Z_1 \cup Z_2}   \longrightarrow \O_{\widetilde{C}}(q^*L_0-E)(-(\alpha-1)(x_1+x_2))  \longrightarrow 0,\]
whence, from \eqref{eq:dimsubloc2}
 we have \begin{eqnarray*}
  \dim (\F) & = &  h^0((q^*L_0-E)  \* \I_{Z_1 \cup Z_2})-2  \\
                                           & = & h^0(\O_{\widetilde{C}}(q^*L_0-E)(-(\alpha-1)(x_1+x_2)))-1 \\
   & = & \deg(\O_{\widetilde{C}}(q^*L_0-E)(-(\alpha-1)(x_1+x_2)))
  \\
                                           & = & (q^*L_0-2E)(q^*L_0-E)-2(\alpha-1)\\
  & = & L_0^2-2\delta_0-2\alpha+2 \\
  & = & \alpha+2\beta-\sum \gamma_i,
\end{eqnarray*}
as desired.
\end{proof}

\noindent
The proof of Proposition \ref{prop:R11} is now complete.
\end{proof}

%
%

\end{document}